\documentclass[a4paper,11pt]{amsart}
\usepackage[plainpages=false]{hyperref}
\usepackage{amsfonts,latexsym,rawfonts,amsmath,amssymb,amsthm,mathrsfs}
\usepackage{amsmath,amssymb,amsfonts,latexsym,lscape,rawfonts}

\def\TT{{\mathbb{T}}}

\def\ZZ{{\mathbb{Z}}}

\def\CC{{\mathbb{C}}}
\def\RR{{\mathbb{R}}} 
 
\def\TT{{\mathbb{T}}}

\newtheorem{thm}{Theorem}[section]

\newtheorem{cor}[thm]{Corollary}
\newtheorem{lemma}[thm]{Lemma}

\newtheorem{prop}[thm]{Proposition}

\begin{document}

\title{Calabi flow on toric varieties with bounded Sobolev constant, I}
\author{Hongnian Huang}
\date{}

\maketitle

\begin{abstract}
Let $(X, P)$ be a toric variety. In this note, we show that the $C^0$-norm of the Calabi flow $\varphi(t)$ on $X$ is uniformly bounded in $[0, T)$ if the Sobolev constant of $\varphi(t)$ is uniformly bounded in $[0, T)$. We also show that if $(X, P)$ is uniform $K$-stable, then the modified Calabi flow converges exponentially fast to an extremal K\"ahler metric if the Ricci curvature and the Sobolev constant are uniformly bounded. At last, we discuss an extension of our results to a quasi-proper K\"ahler manifold.
\end{abstract}

\section{Introduction}
In \cite{ChenHe2, ChenHe3}, Chen and He study the Calabi flow on toric surfaces with bounded Sobolev constant. The study of Calabi flow with bounded Sobolev constant has been also elaborated by Li and Zheng \cite{LZ}. Related work can be found in \cite{St1, St2, ChenWeber, CLW, TV1, TV2, Xu}. Interested readers are encouraged to read these papers and the references therein.

In this note, we study the Calabi flow on toric varieties with bounded Sobolev constant and on general K\"ahler manifolds which are quasi-proper. Our first result is:

\begin{thm}

\label{thm1}
  Let $X$ be a toric variety with Delzant polytope $P$. Let $\varphi(t), ~  0 \le t < T < \infty$ be a one parameter family toric invariant relative K\"ahler potentials satisfying the Calabi flow equation. Suppose that the Sobolev constant of $\varphi(t)$ is uniformly bounded. Then
  \begin{align*}
    |\varphi(t)|_{L^\infty} < C, ~ \forall ~ t \in [0, T),
  \end{align*}
  where $C$ is some constant independent of $t$.
\end{thm}

Suppose $(X, P)$ is uniform $K$-stable, we would like to understand the global convergence of the (modified) Calabi flow. The following result can be compared to the global convergence results in \cite{H1, H2, Sz1}.

\begin{thm}

\label{thm2}
  If $(X, P)$ is uniform $K$-stable, then the modified Calabi flow introduced in \cite{HZ} will converge to an extremal K\"ahler metric exponentially fast if 

  \begin{itemize}
    \item The Calabi flow starts from a toric invariant K\"ahler metric.

    \item The Sobolev constant is uniformly bounded along the flow.

    \item The Ricci curvature is uniformly bounded along the flow.
  \end{itemize}
\end{thm}

For a general manifold, we have:

\begin{thm}
\label{thm3}

  Let $(X, J, \omega)$ be a quasi-proper K\"ahler manifold in the sense of Chen \cite{Chen3}, i.e.,
  there exists a small constant $\delta > 0$ and a constant $C$ such that the Mabuchi energy
  $$
  Ma(\varphi) \ge \delta \int_X \log \frac{\omega_\varphi^n}{\omega^n} ~ \omega_\varphi^n - C.
  $$
  Then $|\varphi(t)|_{L^\infty}$ is uniformly bounded for $t \in [0, T)$ if the Sobolev constant is uniformly bounded in $[0, T)$.
\end{thm}

{\bf Acknowledgment:} The results in this note are motivated by discussions with Vestislav Apostolov. The author would like to thank him for sharing his ideas.  Thanks also go to Yiyan Xu and Kai Zheng for help discussions.

\section{Notations and setup}

\subsection{K\"ahler geometry}

Let $(X, J, \omega)$ be a K\"ahler manifold with complex dimension $n$, where locally 

\begin{align*}
  \omega = \sqrt{-1} g_{i\bar{j}} d z^i \wedge d \bar{z}^j,
\end{align*}
$\left( g_{i\bar{j}} \right)$ is a positive definite Hermitian matrix. The K\"ahler metric is (locally):

$$
g = g_{i\bar{j}} d z^i \otimes d \bar{z}^j.
$$

The set of K\"ahler metrics can be identified with $\mathcal{H} / \RR$, where

\begin{align*}
  \mathcal{H} = \{ \varphi \in C^\infty(X) ~ | ~ \omega_\varphi = \omega + \sqrt{-1}\partial\bar{\partial} \varphi > 0 \}.
\end{align*}

We call $\varphi \in \mathcal{H}$ is a relative K\"ahler potential. The corresponding K\"ahler metric is 

\begin{align*}
  g_\varphi =  \left( g_{i\bar{j}} + \varphi_{i \bar{j}} \right) d z^i \otimes d \bar{z}^j.
\end{align*}

Its volume form is

\begin{align*}
  \omega_\varphi^n = \frac{\left( \sqrt{-1} \right)^n}{n!} \det\left( g_{i\bar{j}} \right) d z^1 \wedge d \bar{z}^1 \wedge \cdots \wedge d z^n \wedge d \bar{z}^n,
\end{align*}
where $n$ is the complex dimension of $X$. Its Ricci and scalar curvature are:

\begin{align*}
  Ric_\varphi = - \partial \bar{\partial} \log \det(g_{\varphi, i\bar{j}}), \quad R_\varphi = - \triangle_\varphi \log \det(g_{\varphi, i\bar{j}}).
\end{align*}

The Calabi flow \cite{Ca1, Ca2} starting from $\omega_\varphi$ is defined as:

\begin{align*}
  \frac{\partial \varphi}{\partial t} = R_\varphi - \underline{R},
\end{align*}

where $\underline{R} = \frac{\int_X R_\varphi ~ \omega_\varphi^n}{\int_X ~ \omega_\varphi^n}$ is a topological constant. The short time existence of the Calabi flow is established in \cite{ChenHe1}.

Now we introduce three functionals $I, J, D$. The $I$ and $J$ functionals are introduced by Aubin \cite{Au}: for any $\varphi \in \mathcal{H}$,

\begin{align*}
  I(\varphi) = & \int_X \varphi (\omega^n - \omega_\varphi^n) = \sqrt{-1} \sum_{i=0}^{n-1} \int_X \partial \varphi \wedge \partial \bar{\varphi} \wedge \omega^i \wedge \omega_\varphi^{n-1-i}, \\
  J(\varphi) = & \sqrt{-1} \sum_{i=0}^{n-1} \frac{i+1}{n+1} \int_X \partial \varphi \wedge \partial \bar{\varphi} \wedge \omega^i \wedge \omega_\varphi^{n-1-i}.
\end{align*}

We have

\begin{align*}
  \frac{1}{n+1}I \le J \le \frac{n}{n+1}I.
\end{align*}

The third functional is the $D$ functional defined as follows: (c.f. \cite{Di})

\begin{align*}
  D(\varphi) = \int_X \varphi ~ \omega^n - J(\varphi).
\end{align*}

Direct calculations show that 

\begin{align*}
  \delta \left( D(\varphi) \right)(f) = \int_X f ~ \omega_\varphi^n.
\end{align*}

Thus $D(t) = const$ along the Calabi flow.

\subsection{Toric geometry}

Let $X$ be a toric manifold and $\omega$ be a toric invariant K\"ahler metric. We obtain the Delzant polytope $P$ of $X$ through the moment map. On $P$, we have the standard Lebesgue measure $d \mu$. On each facet $P_i$ of $P$, the measure $d \sigma$ equals to $\frac{1}{|\vec{n}_i|}$ times the standard Lebesgue measure, where $\vec{n}_i$ is an inward normal vector associated to $P_i$. Also for any vertex $v$ of $P$, there exists exactly $n$ facets $P_1, \ldots, P_n$ intersecting at $v$ and $(\vec{n}_1, \ldots, \vec{n}_n)$ is a basis of $\ZZ^n$.

Suppose that $P$ has $d$ facets. Each facet $P_i$ can be represented by

\begin{align*}
  l_i(x) = 0, 
\end{align*}

where $l_i(x) = \langle x, \vec{n}_i \rangle + c_i$. A symplectic potential $u$ of $P$ satisfies the following Guillemin boundary conditions \cite{Gu}:

\begin{itemize}
  \item $u$ is a smooth, strictly convex function on $P$.

  \item The restriction of $u$ to each facet of $P$ is also a smooth, strictly convex function.

  \item 

    \begin{align*}
      u(x) = \frac{1}{2} \sum_{i=1}^d l_i(x) \ln l_i(x) + f(x), ~ x \in P,
    \end{align*}

    where $f(x)$ is a smooth function on $\bar{P}$.
\end{itemize}

On the open orbit of the $(\CC^*)^n = \RR^n \times \TT^n$ action of $X$, a toric invariant K\"ahler metric $\omega$ can be express as

$$
\omega = \psi_{ij} d z^i \wedge d \bar{z}^j,
$$
where $\psi$ is a real, smooth, strictly convex function on $\RR^n$ and $z_i = \xi_i + \sqrt{-1} \theta_i, ~ (\xi_1, \ldots, \xi_n) \in \RR^n, ~ (\theta_1, \ldots, \theta_n) \in \TT^n$. In fact, the Legendre dual of $\psi$ is a symplectic potential $u$ on $P$. Thus

\begin{align*}
  \psi(\xi) + u(x) = \sum_{i=1}^n \xi_i x_i, \quad x = \nabla \psi(\xi).
\end{align*}

The following Proposition is due to Donaldson \cite{D1}:

\begin{prop}
  For any toric invariant $\varphi \in \mathcal{H}$, we obtain a symplectic potential $u_\varphi$ through the Legendre dual of $\psi + \varphi$. The map from $\varphi$ to $u_\varphi$ is one to one and onto.
\end{prop}

Using the Legendre transform, Abreu \cite{Ab} shows that the scalar curvature equation of a toric invariant metric $\omega_\varphi$ on the open $(\CC^*)n$ orbit 

\begin{align*}
  R_\varphi = \triangle_\varphi \log \det(D^2 (\psi + \varphi))
\end{align*}

can be transformed to

\begin{align*}
  R_{u_\varphi} = - \sum_{i,j=1}^n u_{\varphi, ~ ij}^{ij}
\end{align*}

in the symplectic side. Thus the Calabi flow equation on $P$ is

\begin{align*}
  \frac{\partial u}{\partial t} = \underline{R} - R_u.
\end{align*}

\section{Controlling $\max \varphi$}
Let $u(t, x)$ be a sequence of symplectic potentials satisfying the Calabi flow equation on $P$. Our first lemma is

\begin{lemma}
  \begin{align*}
    \| u(t,x) \|_{L^2} < C(P, T, u_0),
  \end{align*}
  where $C(P, T, u_0)$ is a constant depending on $P, T$ and $u(0, x)$.
\end{lemma}

\begin{proof}
  Since the Calabi flow decreases the distance \cite{CC}, for any $t \in [0, T)$, we have

    \begin{align*}
      \| u(\frac{t}{2}, x) - u(0, x) \|_{L^2} \ge \| u(t,x) - u(\frac{t}{2}, x)\|_{L^2}.
    \end{align*}

  By the triangle inequality, we have

    \begin{align*}
      2 \| u(\frac{t}{2}, x) \|_{L^2} + \| u(0, x) \|_{L^2} \ge \| u(t,x) \|_{L^2}.
    \end{align*}
  
  Similarly, 

  \begin{align*}
    2 \| u(\frac{t}{4}, x) \|_{L^2} + \| u(0, x) \|_{L^2} \ge \| u(\frac{t}{2},x) \|_{L^2}.
    \end{align*}

Thus,

    \begin{align*}
      4 \| u(\frac{t}{4}, x) \|_{L^2} + 3 \| u(0, x) \|_{L^2} \ge \| u(t,x) \|_{L^2}.
    \end{align*}

    It is easy to see that there exists a constant $C(P, T, u_0)$ depending on $P, T$ and $u(0, x)$ such that

  \begin{align*}
    \| u(t,x) \|_{L^2} < C(P, T, u_0),
  \end{align*}
\end{proof}

An immediate corollary is:

\begin{cor}
  $$ \min_{x \in \bar{P}} u(t,x) < C_1(P, T, u_0), $$
  where $C_1(P, T, u_0)$ is some constant depending on $P, T$ and $u(0, x)$.
\end{cor}

Since the $L^2$ norm of $u(t, x)$ is bounded by $C(P, T, u_0)$, our next lemma shows that $\min_{x \in \bar{P}} u(x)$ is bounded from below by some constant $C_2(P, T, u_0)$ which depends on $u(0, x), P$ and $T$.

\begin{lemma}
  If $\| u(t, x) \|_{L^2} < C(P, T, u_0)$ for any $t < T$, then
  \begin{align*}
    \min_{x \in \bar{P}} u(x, t) > C_2(P, T, u_0)
  \end{align*}
\end{lemma}

\begin{proof}
  To simplify our notation, we write $u(t, x)$ as $u(x)$. Let $o$ be the barycenter of $P$. By shrinking the vertices by $\frac{1}{2}$ around $o$, we get $P_{\frac{1}{2}}$. Our first observation is that there exists a constant $C_3(P, T, u_0)$ depending on $P, T$ and $u(0, x)$ such that for any $x \in P_{\frac{1}{2}}$, we have

  \begin{align*}
    u(x) < C_3(P, T, u_0).
  \end{align*}
  
  This is because:
  \begin{itemize}
    \item $u(x)$ is a convex function.

    \item For any $x \in P_{\frac{1}{2}}$, any hyperplane $l$ passing through $x$ will cut $P$ as $P_1$ and $P_2$. There exists a positive constant $C_4(P)$ depending on $P$ such that the Euclidean volume of $P_1$ and $P_2$ are both greater than $C_4(P)$.
  \end{itemize}

  Let $x_0 \in \bar{P}$ be a point such that $u(x_0) = \min_{x \in \bar{P}} u(x)$. Then there exists a constant $C_5(P)$ depending on $P$ such that the follwing holds:
  \begin{itemize}
    \item 
  	At least one facet of $P_{\frac{1}{2}}$, say $Q \subset \partial P_{\frac{1}{2}}$ such that the Euclidean distance between $x_0$ and $Q$ is greater than $C_5(P)$.
  \end{itemize}

  By shrinking the vertices of $Q$ by $\frac{1}{2}$ around $x_0$, we obtain $Q_{\frac{1}{2}}$. We denote $\tilde{P}$ as the convex hull of $x_0$ and $Q_{\frac{1}{2}}$. Then for any point $x \in \tilde{P}$, we have
  \begin{align*}
    u(x) \le \frac{u(x_0) + C_3(P, T, u_0)}{2}.
  \end{align*}
  Since we know
  \begin{align*}
    \int_{\tilde{P}} u^2(x) ~ d \mu \le \int_P u^2(x) ~ d \mu < C_1(P, T, u_0)^2,
  \end{align*}
  it is clear that there exists a constant $C_2(P, T, u_0)$ depending on $P$, $T$ and $u(0, x)$ such that
  \begin{align*}
    \min_{x \in \bar{P}} u(x, t) > C_2(P, T, u_0).
  \end{align*}
\end{proof}

Translating our result to the complex side, we have the following proposition:

\begin{prop}
  \label{upper_bound_max}
  There exists a constant $C_3(P, T, u_0)$ depending on $P, T$ and $u(0, x)$ such that for any $t < T$ we have
  \begin{align*}
    \max_{z \in X} \varphi(t, z) < C_3(P, T, u_0).
  \end{align*}
\end{prop}

\begin{proof}
  We write $x = \nabla(\psi(t, \xi)), ~ \xi \in \RR^n$. Then we have

  \begin{align*}
    \varphi(t, \xi) & = \psi(t, \xi) - \psi(0, \xi)  \\
    & = \left( \sum_{i=1}^n \xi x - u(t, x) \right) - \max_{y \in P} \left( \sum_{i=1}^n \xi y - u(0, y) \right) \\
    & \le u(0, x) - u(t, x) \\
    & < C_3(P, T, u_0).
  \end{align*}
\end{proof}

A conjecture of Donaldson, proved by Chen \cite{Chen1}, saying that 

\begin{lemma}
  \begin{align*}
    d(0, \varphi) \ge \frac{1}{\sqrt{C(P)}} \left( \max \left( \int_{\varphi > 0} \varphi ~ \omega_{\varphi}^n, ~ - \int_{\varphi < 0} \varphi ~ \omega^n \right) \right),
  \end{align*}
  where $C(P) = (2 \pi)^n Vol(P)$.
\end{lemma}

In the toric case, we can have a stronger result:

\begin{lemma}
\label{two_integral_bound}
  For any toric invariant K\"ahler metric
  \begin{align*}
    \omega_\varphi = \omega + \sqrt{-1} \partial \bar{\partial} \varphi = \sqrt{-1} \partial \bar{\partial} (\psi + \varphi) ~ \mathrm{on} ~ \RR^n.
  \end{align*}
  Let the Legendre dual of $\psi, \psi + \varphi$ be $u, u_\varphi$ respectively. We have
  \begin{align*}
    \frac{(2\pi)^n}{n!} \int_P (u_\varphi - u)^2 ~ d \mu \ge \left( \max \left( \int_{\varphi > 0} \varphi^2 ~ \omega_{\varphi}^n, ~ \int_{\varphi < 0} \varphi^2 ~ \omega^n \right) \right).
  \end{align*}
\end{lemma}

\begin{proof}
  Let $x = \nabla(\psi + \varphi)(\xi), ~ \xi \in \RR^n$. Again we have

  \begin{align*}
    \varphi(\xi) \le u(x) - u_\varphi(x).
  \end{align*}

  It implies that

  \begin{align*}
    \frac{(2\pi)^n}{n!} \int_P (u_\varphi - u)^2 ~ d \mu \ge \int_{\varphi > 0} \varphi^2 ~ \omega_{\varphi}^n.
  \end{align*}

  On the other hand, let $x = \nabla \psi(\xi), ~ \xi \in \RR^n$. Similarly we have

  \begin{align*}
    \varphi(\xi) & = (\psi + \varphi)(\xi) - \psi(\xi)  \\
    & = \left( \max_{y \in P}\sum_{i=1}^n \xi y - u_\varphi(y) \right) - \left( \sum_{i=1}^n \xi x - u(x) \right) \\
    & \ge u(x) - u_\varphi(x).
  \end{align*}

  Then we have

  \begin{align*}
    \frac{(2\pi)^n}{n!} \int_P (u_\varphi - u)^2 ~ d \mu \ge \int_{\varphi < 0} \varphi^2 ~ \omega^n.
  \end{align*}
\end{proof}

As a consequence, we have 

\begin{cor}

  \label{cor_max}
  There exists a constant $C_4(P, T, u_0)$ depending on $P, T$ and $u(0, x)$ such that for any $t < T$ we have

  \begin{align*}
    \max_{z \in X} \varphi(t, z) > C_4(P, T, u_0).
  \end{align*}
\end{cor}

\section{$L^1$-norm}

It is easy to see that for any $t < T$, we have

\begin{lemma}
\label{lem_integral_bound}
  \begin{align*}
    \int_X |\varphi(t)| ~ \omega^n < C_5(P, T, u_0),
  \end{align*}
  for some constant $C_5(P, T, u_0)$ depending on $P, T$ and $u_0$.
\end{lemma}

In fact, it is well known that the bound of

\begin{align*}
  \int_X |\varphi| ~ \omega^n
\end{align*}

can be derived from the bound of $\max_{z \in X} \varphi(z)$. The arguments go as follows: \\

Let $z_0 \in X$ be the point where $\varphi(z)$ reaches its maximum. By the Green's formula, we have

\begin{align*}
  \varphi(z_0) = \left( Vol(X) \right)^{-1} \int_X \varphi ~ \omega^n - \left( Vol(X) \right)^{-1} \int_X \triangle_\omega \varphi(\tilde{z}) G(z_0, \tilde{z}) ~ \omega^n.
\end{align*}

Since $\omega_\varphi = \omega + \sqrt{-1} \partial \bar{\partial} \varphi > 0$, taking trace with respect to $\omega$, we have

\begin{align*}
  n + \triangle_\omega \varphi > 0.
\end{align*}

Thus we obtain
$$
\varphi(z_0) - \left( Vol(X) \right)^{-1} \int_X \varphi ~ \omega^n < C(\omega),
$$
where $C(\omega)$ is a constant depending on $\omega$.

\begin{itemize}

  \item

If $\varphi(z_0) > 0$, then

\begin{align*}
  C(\omega) & > \varphi(z_0) - Vol(X)^{-1} \int_{\varphi > 0} \varphi ~ \omega^n + Vol(X)^{-1} \int_{\varphi < 0} -\varphi ~ \omega^n \\
  & > Vol(X)^{-1} \int_{\varphi < 0} -\varphi ~ \omega^n.
\end{align*}

Thus we have 

\begin{align}
  \label{integral_bound}
  \int_X |\varphi| ~ \omega^n < C(\omega, \varphi(z_0)),
\end{align}

where $C(\omega, \varphi(z_0))$ depending on $\omega$ and $\varphi(z_0)$.

\item If $\varphi(z_0) \le 0$, then it is straightforward to see inequality \eqref{integral_bound}.

\end{itemize}

Our next lemma shows that we can also bound the $L^1$-norm of $\varphi(t)$ with respect to $\omega(t)$.

\begin{lemma}
\label{lem_integral_bound_2}
  Suppose that $D_\omega(\varphi) = 0$, then
  \begin{align}
    \label{integral_bound_2}
    \int_X |\varphi| ~ \omega_\varphi^n < C \left(\max_{z \in X} \varphi(z), |\varphi|_{L^1(\omega)} \right),
  \end{align}
  where $C\left(\max_{z \in X} \varphi(z), |\varphi|_{L^1(\omega)} \right)$ is some constant depending on $\max_{z \in X} \varphi(z)$ and $|\varphi|_{L^1(\omega)}$.
\end{lemma}

\begin{proof}
  Recall that

  \begin{align*}
    D(\varphi) = \int_X \varphi ~ \omega^n - J(\varphi).
  \end{align*}

  The bound of $L^1_\omega(\varphi)$ give us the bound of $J(\varphi)$. Since
\begin{align*}
  \frac{1}{n+1}I(\varphi) \le J(\varphi) \le \frac{n}{n+1} I(\varphi).
\end{align*}

We obtain the bound of $I(\varphi)$ which also gives us the bound of 

\begin{align*}
  \int_X \varphi ~ \omega_\varphi^n.
\end{align*}

Hence we obtain the inequality \eqref{integral_bound_2}.

\end{proof}

\section{$L^\infty$-norm}

First, we will try to control the $L^2$-norm of $\varphi$ with respect of $\omega_\varphi$:

\begin{prop}
\label{prop_L2}
  Let $\varphi$ be a smooth relative K\"ahler potential with $|\varphi(z)| \ge c_0 > 0, ~ \forall z \in X$. Then
  \begin{align*}
    \|\varphi\|_{L^2_{\omega_\varphi}} \le C(n, c_0, C_s, J_\omega(\varphi), L^1_{\omega_\varphi}(\varphi)),
  \end{align*}
  where $C(n, c_0, C_s, J_\omega(\varphi), L^1_{\omega_\varphi}(\varphi))$ is a constant depending on $n, ~ c_0$, the Sobolev constant of $\omega_\varphi, ~ J_\omega(\varphi)$ and the $L^1$-norm of $\varphi$ with respect to $\omega_\varphi$.
\end{prop}

\begin{proof}
  Notice that 
  \begin{align*}
    &\int_X |\nabla \sqrt{|\varphi|}|^2 ~ \omega_\varphi^n \\
    = & \int_X \frac{|\nabla \varphi|^2}{4|\varphi|} ~ \omega_\varphi^n \\
    \le & \frac{1}{4 c_0} \int_X |\nabla \varphi|^2 ~ \omega_\varphi^n \\
    \le & \frac{n+1}{4 c_0} J_\omega(\varphi).
  \end{align*}

  The Sobolev inequality shows that

  \begin{align*}
    \| \sqrt{|\varphi|} \|_{L^{\frac{2n}{n-1}}_{\omega_\varphi}} \le C_s\left( \| \nabla \sqrt{|\varphi|} \|_{L^2_{\omega_{\varphi}}} + \| \sqrt{|\varphi|} \|_{L^2_{\omega_{\varphi}}} \right)
  \end{align*}

  It is clear that

  \begin{align*}
    \| \varphi \|_{L^{\frac{n}{n-1}}_{\omega_\varphi}} \le C(n, c_0, C_s, J_\omega(\varphi), L^1_{\omega_\varphi}(\varphi)).
  \end{align*}

  If the complex dimension of $X$ is $2$, then we are done. If $n > 2$, then we let $f = |\varphi|^{\frac{n}{2(n-1)}}$. Notice that
  \begin{align*}
  |\nabla f| &= \frac{n}{2(n-1)} \frac{|\nabla \varphi|}{|\varphi|^{\frac{n-2}{2(n-1)}}} \le C(n, c_0) |\nabla \varphi|\\
  \| f \|_{L^2_{\omega_{\varphi}}} &= \left( \| \varphi \|_{L^{\frac{n}{n-1}}_{\omega_\varphi}} \right)^{\frac{n}{2(n-1)}} \le C(n, c_0, C_s, J_\omega(\varphi), L^1_{\omega_\varphi}(\varphi)).
  \end{align*}

  Applying the Sobolev inequality again, we have

  \begin{align*}
    \| \varphi \|_{L^{\frac{n^2}{(n-1)^2}}_{\omega_\varphi}} =  \left( \| f \|_{L^{\frac{2n}{n-1}}_{\omega_\varphi}} \right)^{\frac{2(n-1)}{n}} \le C(n, c_0, C_s, J_\omega(\varphi), L^1_{\omega_\varphi}(\varphi)).
  \end{align*}

  Repeating the steps, we get the conclusion.
\end{proof}

Once we have the $L^2$ estimate, we could get the $L^\infty$ estimate by using the De Giorgi-Nash-Moser iteration:

\begin{prop}
\label{prop_C0}
  For a smooth relative K\"ahler potential $\varphi$,

  \begin{align*}
    |\varphi|_{L^\infty} \le C\left(n, C_s, \max_{z \in X}\varphi(z), L^2_{\omega_\varphi}(\varphi)\right),
  \end{align*}
  where $C\left(n, C_s, \max_{z \in X}\varphi(z), L^2_{\omega_\varphi}(\varphi)\right)$ is a constant depending on $n$, the Sobolev constant of $\omega_\varphi$ , $\max_{z \in X}\varphi(z)$ and $L^2_{\omega_\varphi}(\varphi))$.
\end{prop}

\begin{proof}
  Let $\varphi_1 = - \varphi + \max_{z \in X} \varphi(z) + 1$.
  Since $\omega= \omega_\varphi + \sqrt{-1} \partial \bar{\partial} \varphi_1 > 0$, we have
  \begin{align*}
    n + \triangle_\varphi \varphi_1 > 0.
  \end{align*}

  Thus for any $p \ge 1$,

  \begin{align*}
    \int_X n \varphi_1^p ~ \omega_\varphi^n \ge & \int_X - \varphi_1^p \triangle_\varphi \varphi_1 ~ \omega_\varphi^n \\
    = & \frac{p}{(\frac{p+1}{2})^2} \int_X |\nabla \varphi_1^{\frac{p+1}{2}}|^2_\varphi ~ \omega_\varphi^n \\
  \end{align*}

  Thus by the Sobolev inequality, we have

  \begin{align*}
    \left( \int_X \varphi_1^{(p+1) \frac{n}{n-1}} ~ \omega_\varphi^n \right)^{\frac{n-1}{n}} \le C(n, C_s) (p+1) \int_X \varphi_1^{p+1} ~ \omega_\varphi^n.
  \end{align*}

  Then

  \begin{align*}
    \| \varphi_1 \|_{L^{\frac{(p+1)n}{n-1}}_{\omega_\varphi}} \le \left( C (p+1)\right)^{\frac{1}{p+1}} \|\varphi_1 \|_{L^{p+1}_{\omega_\varphi}}.
  \end{align*}

  By iterations, we get the conclusion.
\end{proof}

\section{Proofs of the theorems}

\begin{proof}[Proof of Theorem (\ref{thm1})]

By Proposition (\ref{upper_bound_max}) and Corollary (\ref{cor_max}), we uniformly control $\max_{z \in X} \varphi(t, z)$ for any $t \in [0, T)$. Then Lemma (\ref{lem_integral_bound}) and Lemma (\ref{lem_integral_bound_2}) provide us the uniform $L^1$-norm bound of $\varphi(t)$ with respect to $\omega$ and $\omega(t)$ respectively. Hence Proposition (\ref{prop_L2}) gives us the $L^2$-norm of $\varphi(t)$ with respect to $\omega(t)$ uniformly and Proposition (\ref{prop_C0}) gives us the $L^\infty$-norm of $\varphi(t)$ uniformly.
\end{proof}

\subsection{Global convergence}
\begin{proof} [Proof of Theorem (\ref{thm2})]
  Let us fix $t_0 > 0$ and write $\varphi(t)$ as $\varphi_0$ and $u(t)$ as $u_0$. Let $u$ be the normalized symplectic potential of $u_0$ at some point $x_0 \in P$. Since $(X, P)$ is uniformly $K$-stable, we have $\int_P u ~ d \mu < C$ where $C$ is some constant independent of $t$ by Proposition 5.1.8 and Lemma 5.1.3 in \cite{D1}. \\

We shall consider the corresponding K\"ahler potential $\psi$ of $u$ under the Legendre transformation and the relative K\"ahler potential $\varphi = \psi - \psi_\omega$. As in Proposition (\ref{upper_bound_max}) and Corollary (\ref{cor_max}), we obtain the upper and lower bound of $\max_{z \in X} \varphi(z)$. By using the arguments of Lemma (\ref{two_integral_bound}), we obtain the bounds of

\begin{align*}
  \int_{\varphi > 0} \varphi ~ \omega_\varphi^n, \quad \int_{\varphi < 0} -\varphi ~ \omega^n.
\end{align*}

Lemma 2.1 in \cite{ZZ} provides the bound of $J(\varphi)$. Thus we obtain the bound of 
\begin{align*}
  \int_X |\varphi| ~ \omega_\varphi^n.
\end{align*}
As before, we obtain the $L^\infty$ estimate of $\varphi$. The rest of the proof is identical to the proof of Theorem 1.6 of \cite{H2}.
\end{proof}

\subsection{General case}

\begin{proof} [Proof of Theorem (\ref{thm3})]
  By the proof of Theorem 1.4 in \cite{Chen3}, we obtain uniform upper bounds of 

  \begin{align*}
    \max_{z \in X} \varphi(t, z)
  \end{align*}
  and
  \begin{align*}
    \int_X |\varphi| ~ \omega^n.
  \end{align*}

   Then Lemma (\ref{lem_integral_bound_2}) provide us the $L^1$-norm bound of $\varphi(t)$ with respect to $\omega(t)$ uniformly. Hence Proposition (\ref{prop_L2}) gives us the $L^2$-norm of $\varphi(t)$ with respect to $\omega(t)$ uniformly and Proposition (\ref{prop_C0}) gives us the $L^\infty$-norm of $\varphi(t)$ uniformly.
\end{proof}


\begin{thebibliography}{10}
\bibitem{Ab}, M. Abreu, {K\"ahler geometry of toric varieties and extremal metrics}, Int. J. Math. 9,
641--651 (1998).

\bibitem{Au} T. Aubin, {R\'eduction du cas positif de l'\'equation de Monge-Amp\`ere sur les vari\'et\'es
k\"ahl\'eriennes compactes \`a la d\'emonstration d'une in\'egalit\'e}, J. Funct. Anal. 57 (1984), no. 2,
143--153.

\bibitem{Ca1} E. Calabi, {\em Extremal K\"ahler metric, in Seminar of Differential Geometry}, ed. S. T. Yau, Annals of Mathematics Studies 102, Princeton University Press (1982), 259--290.

\bibitem{Ca2} E. Calabi, {\em Extremal K\"ahler metric, II, in Differential Geometry and Complex Analysis}, eds. I. Chavel and H. M. Farkas, Spring Verlag (1985), 95--114.

\bibitem{CC} E. Calabi and X.X. Chen, {\em Space of K\"ahler metrics and Calabi flow}, J. Differential Geom. 61 (2002), no. 2, 173--193.

\bibitem{Chen1} X.X. Chen, {\em The space of K\"ahler metrics}, J. Differential Geom. 56 (2000), no. 2, 189--234. 

\bibitem{Chen3} X.X. Chen, {\em Space of K\"ahler metrics. III. On the lower bound of the Calabi energy and geodesic distance}, Invent. Math. 175 (2009), no. 3, 453--503.

\bibitem{ChenHe1} X.X. Chen and W.Y. He, {\em On the Calabi flow},  Amer. J. Math. 130 (2008), no. 2, 539--570.

\bibitem{ChenHe2} X.X. Chen and W.Y. He, {\em The Calabi flow on K\"ahler surface with bounded Sobolev constant (I)}, Math. Ann. 354 (2012), no. 1, 227--261.

\bibitem{ChenHe3} X.X. Chen and W.Y. He, {\em The Calabi flow on toric Fano surfaces}, Math. Res. Lett. 17 (2010), no. 2, 231--241.

\bibitem{CLW} X.X. Chen, C. LeBrun and B. Weber, {\em On conformally K\"ahler, Einstein manifolds}, J. Amer. Math. Soc. 21 (2008), no. 4, 1137--1168.

\bibitem{ChenWeber} X.X. Chen and B. Weber, {\em Moduli spaces of critical Riemannian metrics with $L^\frac{n}{2}$ norm curvature bounds}, Adv. Math. 226 (2011), no. 2, 1307--1330.

\bibitem{Di} W.Y. Ding, {\em Remarks on the existence problem of positive K\"ahler-Einstein metrics}, Math.
Ann., 282(3):463--471, 1988.

\bibitem{D1} S.K. Donaldson, {\em Scalar curvature and stability of toric varieties}, Jour. Differential Geometry 62 (2002), 289--349.

\bibitem{Gu} V. Guillemin, {\em Kaehler structures on toric varieties}, J. Diﬀer. Geom. 40, 285--309
(1994).

\bibitem{H1} H.N. Huang, {\em Toric Surface, $K$-Stability and Calabi Flow}, Math. Z. 276 (2014), no. 3-4, 953--968.

\bibitem{H2} H.N. Huang, {\em Convergence of the Calabi flow on toric varieties and related K\"ahler manifolds}, to appear in J. Geom. Anal.

\bibitem{HZ} H.N. Huang and K. Zheng, {\em Stability of Calabi flow near an extremal metric},  Ann. Sc. Norm. Super. Pisa Cl. Sci. (5) 11 (2012), no. 1, 167--175.

\bibitem{LZ} H.Z. Li and K. Zheng, {\em K\"ahler non-collapsing, eigenvalues and the Calabi flow}, to appear in J. Funct. Anal. 

\bibitem{St1} J. Streets, {\em Collapsing in the $L^2$ curvature flow}, Comm. PDE. 38: 985--1014, 2013.

\bibitem{St2} J. Streets, {\em A concentration-collapse decomposition for $L^2$ flow singularities}, arXiv:1311.0961.

\bibitem{Sz1} G. Sz\'ekelyhidi, {\em Remark on the Calabi flow with bounded curvature}, Univ. Iagel. Acta Math. 50 (2013), 107--115.

\bibitem{TV1} G. Tian and J.A. Viaclovsky, {\em Bach-flat asymptotically locally Euclidean metrics}, Invent. Math. 160 (2) (2005) 357--415.

\bibitem{TV2} G. Tian and J.A. Viaclovsky, {\em  Moduli spaces of critical Riemannian metrics in dimension four}, Adv. Math. 196 (2005), no. 2, 346--372.

\bibitem{Xu} Y.Y. Xu, {\em Regularity for the harmonic-Einstein equation}, Adv. Math. 231 (2012), no. 2, 680--708. 

\bibitem{ZZ} B. Zhou and X.H. Zhu, {\em Relative K-stability and modified K-energy on toric manifolds}, Adv. Math. 219 (2008), no. 4, 1327--1362.
\end{thebibliography}
\end{document}